\documentclass[12pt,reqno]{amsart}
\usepackage{graphicx} 
\usepackage{fancyhdr}
\usepackage{latexsym, amsmath, enumerate, color, amssymb,amsthm, amsfonts, hyperref,mathtools, rotating, amscd, enumerate, multirow, makecell}
\allowdisplaybreaks[4]
\newtheorem{theorem}{Theorem}
\newtheorem{corollary}[theorem]{Corollary}
\newtheorem{lemma}[theorem]{Lemma}

\theoremstyle{definition} 
\newtheorem{definition}[theorem]{Definition}

\newtheorem{remark}[theorem]{Remark}
\title{Formulas for the Generalized Frobenius Number of Triangular Numbers}
\author{Kittipong Subwattanachai}
\date{\today}

\begin{document}
	
	\maketitle
	
	\begin{abstract}
		For $ k \geq 2 $, let $ A = (a_{1}, a_{2}, \ldots, a_{k}) $ be a $k$-tuple of positive integers with $\gcd(a_{1}, a_2, \ldots, a_k) = 1$. For a non-negative integer $s$, the generalized Frobenius number of $A$, denoted as $\mathtt{g}(A;s) = \mathtt{g}(a_1, a_2, \ldots, a_k;s)$, represents the largest integer that has at most $s$ representations in terms of $a_1, a_2, \ldots, a_k$ with non-negative integer coefficients.
		In this article, we provide a formula for the generalized Frobenius number of three consecutive triangular numbers, $\mathtt{g}(t_{n}, t_{n+1}, t_{n+2};s) $, valid for all $s \geq 0$ where $t_n$ is given by $\binom{n+1}{2}$. Furthermore, we present the proof of Komatsu's conjecture from \cite{Komatsu-triangular}.
	\end{abstract}
	
	\section{Introduction}
        The purpose of this paper is to provide a confirmation of Komatsu's conjecture from \cite{Komatsu-triangular}. He proposed a conjecture for the explicit formula for generalized Frobenius numbers of three consecutive triangular numbers. Before reaching the conclusion, we will introduce some notations and background informatio. Afterward, we will present the full statement of Komatsu's conjecture.  
        
	Let $ A = (a_{1}, a_{2}, \ldots, a_{k}) $ be a $k$-tuple of positive integers with $ \gcd(A) = \gcd(a_{1}, a_2, \ldots, a_k) =1$. We let, for $n \in\mathbb{Z}_{\geq 0}$,
	 $\operatorname{d}(n;A) = \operatorname{d}(n;a_{1}, a_{2}, \ldots, a_{k})$ be the number of representations to $a_{1}x_{1} + a_{2}x_{2} + \cdots + a_{k}x_{k} = n$. Its generating series is given by
	\begin{equation*}
		\sum_{n\geq 0} \operatorname{d}(n;a_1,\dots,a_k) x^n = \frac{1}{(1-x^{a_{1}})(1-x^{a_{2}})\cdots(1-x^{a_{k}})}.
	\end{equation*}
    The study of the generalized Frobenius number has gained attention due to its connection to number theory and combinatorial structures. 
    Sylvester \cite{Sylvester-On_the_partition_of_numbers} and Cayley \cite{Cayley:double_partition} demonstrated that the function $\operatorname{d}(n; a_1, a_2, \ldots, a_k)$ can be expressed as the sum of two components: a polynomial in $n$ of degree $k-1$ and a periodic function with a period of $a_1 a_2 \cdots a_k$. 
    Beck, Gessel, and Komatsu \cite{Beck_Gessel_Komatsu} refined this result by providing an explicit formula for the polynomial component using Bernoulli numbers. 
    For the two-variable case, Tripathi \cite{Tripathi-The_number_of_solutions} derived a specific formula for $\operatorname{d}(n; a_1, a_2)$. Extending to three variables, Komatsu \cite{Komatsu-On_the_number_of_solutions} showed that the periodic component can be expressed using trigonometric functions when the integers $a_1, a_2,$ and $a_3$ are pairwise coprime. 
    Additionally, Binner \cite{Binner-number of solutions to axby+cz=n} derived a formula for the number of non-negative integer solutions to the equation $ax + by + cz = n$ and established a connection between the number of solutions and quadratic residues.

	
	
	We remind readers of the well-known linear Diophantine problem introduced by Sylvester, known as the \emph{Frobenius problem}\footnote{It is also known as the coin problem, postage stamp problem, or Chicken McNugget problem.}:
    ``Given positive integers $a_{1}, a_{2}, \ldots, a_{k}$ with they are relatively prime,
    find the largest integer that \emph{cannot} be expressed as a non-negative integer linear combination of these numbers.''
	\emph{The largest integer} is called the \emph{Frobenius number} of the tuple $A = (a_{1}, a_{2}, \ldots, a_{k})$, and is denoted by $\mathtt{g}(A) = \mathtt{g}(a_{1}, a_{2}, \ldots, a_{k})$. 
	Then, with the above notation, the Frobenius number is given by 
	\begin{equation*}
		\mathtt{g}(A) = \max \{n \in \mathbb{Z} \mid \operatorname{d}(n;A) = 0\}.
	\end{equation*}
	Note that if all non-negative integers can be expressed as a non-negative integer linear combination of $A$, then $\mathtt{g}(A)= -1$. For example, $\mathtt{g}(1,2) = -1$.
	
	For two variables $A = (a,b) \subset \mathbb{Z}_{>0}$, it is shown by Sylvester \cite{Sylvester-On_subinvariants} that
	\begin{equation}\label{eq:twovariableclassical}
		\mathtt{g}(a,b) = ab-a-b.
	\end{equation}
	For instance, if $A = (a,b) = (7,11)$, then the Frobenius number of $A$ is given by $\mathtt{g}(7,11) = 77-7-11 = 59$. It means that all integers $n > 59$ can be expressed as a non-negative integer linear combination of $7$ and $11$. 
	
    The formulas for calculating the Frobenius number in three variables has been provided by Tripathi \cite{Tripathi-Formulae_for_Frobenius_three_variables}. 
    Unfortunately, it is important to recognize that deriving closed-form solutions becomes more challenging as the number of variables increases beyond three $(k > 3)$. 
    Despite these difficulties, various formulas have been developed for Frobenius numbers in specific contexts or particular cases. For example, explicit formulas have been established for certain sequences, including arithmetic, geometric-like, Fibonacci, Mersenne, and triangular numbers (see \cite{Roble_Rosales} and references therein).
\\
	
	In this work, we focus on a generalization of the Frobenius number. For a given integer $s\geq 0$, let $$\mathtt{g}(A;s) = \mathtt{g}(a_1, a_2, \ldots, a_k;s) = \max\{ n \in \mathbb{Z} \mid \operatorname{d}(n;A) \leq s\}$$ be the largest integer such that the number of expressions that can be represented by $a_1, a_2, \ldots, a_k$ is at most $s$ ways. That means all integers bigger than $\mathtt{g}(A;s)$ have at least $s+1$ representations. The $\mathtt{g}(A;s)$ is called \emph{the generalized Frobenius number}. Furthermore, $\mathtt{g}(A;s)$ is well-defined (i.e. bounded above) (see \cite{Fukshansky-Schurmann-Bounds_on}).  Notice that $\mathtt{g}(a_1, a_2, \ldots, a_k;0) = \mathtt{g}(a_1, a_2, \ldots, a_k) $.
	
	As a generalization of \eqref{eq:twovariableclassical}, for $A = (a,b)$ and $s \in\mathbb{Z}_{\geq0}$, (see \cite{Beck_Robins}), an exact formula for $\mathtt{g}(A,s) = \mathtt{g}(a,b;s)$ is given by 
	\begin{equation*}
		\mathtt{g}(a,b;s) = (s+1)ab-a-b. 
	\end{equation*}
    In general, $\operatorname{d}\!\big( \mathtt{g}(A, s); A \big) \leq s$, but for the case of two variables, by Theorem 1 in Beck and Robins \cite{Beck_Robins}, one can see that $\operatorname{d}\!\big( \mathtt{g}(A, s); A \big) = s$ when $|A| = 2$. 
    While exact formulas for the generalized Frobenius number in cases where $k \geq 3$ remain unknown, specific results exist for $k = 3$ in particular cases. Examples include explicit formulas for triangular numbers \cite{Komatsu-triangular}, repunits \cite{Komatsu-The Frobenius number repunits}, and Fibonacci numbers \cite{Komatsu-Ying-Frob_Fibonacci}. Recently, Binner \cite{Binner-binner2021bounds} derived bounds for the number of solutions to the equation $a_1 x_1 + a_2 x_2 + a_3 x_3 = n$ and leveraged these bounds to solve $\mathtt{g}(a_1, a_2, a_3; s)$ for large values of $s$. Woods \cite{Woods-woods2022generalized} also provided formulas and asymptotic results for the generalized Frobenius problem by utilizing the restricted partition function.
	
	One of the special cases is to calculate the Frobenius number for three consecutive triangular numbers, where the $n$th triangular number $t_{n}$ is defined as $t_{n} = \frac{n(n+1)}{2}$ for $n \geq 1$. The explicit formula for $s=0$ (the original Frobenius number) is provided by Robles-P\'{e}rez and Rosales \cite{Roble_Rosales}
	\begin{equation*}
		\mathtt{g} (t_{n}, t_{n+1}, t_{n+2};0)
		=
		\begin{cases}
			\frac{(n+1)(n+2)}{4}(3n) -1,
			&\text{for even } n,
			\\
			\frac{(n+1)(n+2)}{4}(3n-3) -1,
                &\text{for odd } n,
		\end{cases}
	\end{equation*}
	and, for $s = 1,2, \ldots, 10$, they are presented by Komatsu \cite{Komatsu-triangular}. Fuethermore, Komatsu also formulates a conjecture for the explicit formula for $s\geq 0$ as follows;
    
	\textit{``For some non-negative integer $s$, there exists an odd integer $q$ and integers $n_{j}\, (j=1,2,3,4)$ such that
	\begin{equation*}
		\mathtt{g}\left(t_{n}, t_{n+1}, t_{n+2} ; s\right)+1
		= \begin{cases}
			\frac{(q n)(n+1)(n+2)}{4}, 
                &\text{for even } n \geq n_{2},
			\\
			\frac{(q n-3)(n+1)(n+2)}{4}, 
                &\text{for odd } n \geq n_{1},
		\end{cases}
	\end{equation*}
	and
	\begin{equation*}
		\mathtt{g}\left(t_{n}, t_{n+1}, t_{n+2} ; s+1\right)+1
		= 
		\begin{cases}
			\frac{(q n+6)(n+1)(n+2)}{4}, 
                &\text{ for even } n \geq n_{4},
			\\
			\frac{(q n+3)(n+1)(n+2)}{4}, 
                &\text{ for odd } n \geq n_{3},
		\end{cases}
	\end{equation*}
	and so on. 
	For some non-negative integer $s^{\prime}$, there exists an even integer $q^{\prime}$ and integers $n_{5}$ and $n_{6}$ such that
	\begin{equation*}
		\mathtt{g}\left(t_{n}, t_{n+1}, t_{n+2} ; s^{\prime}\right)+1
		=
		\frac{\left(q^{\prime} n\right)(n+1)(n+2)}{4} \quad\left(n \geq n_{5}\right)
	\end{equation*}	
	and
	\begin{equation*}
		\mathtt{g}\left(t_{n}, t_{n+1}, t_{n+2} ; s^{\prime}+1\right)+1
		=
		\frac{\left(q^{\prime} n+6\right)(n+1)(n+2)}{4} \quad\left(n \geq n_{6}\right)
	\end{equation*}
	and so on."} 
    In other words, for many integers $s \geq 0$ (or $s^{\prime}\geq 0$), we can determine $\mathtt{g}(t_{n}, t_{n+1}, t_{n+1}; s+1)$ from $\mathtt{g}(t_{n}, t_{n+1}, t_{n+1}; s)$ and the difference between those two is $\frac{6(n+1)(n+2)}{4}$.
    
    We make this precise in our main results, which are summarized in Theorem \ref{thm-main-formula} and Theorem \ref{thm-different-of-g}. Theorem \ref{thm-main-formula} presents an explicit formula for the generalized Frobenius number involving three consecutive triangular numbers for all $s \geq 0$. Theorem \ref{thm-different-of-g} confirms Komatsu's conjecture and provides a precise statement for the above mentioned phenomena.

	\begin{theorem}\label{thm-main-formula}
		The $\mathtt{g}(t_{n}, t_{n+1}, t_{n+2}; s)$ are given for all $s\geq 0$ as follows:
		\begin{enumerate}[(i)]
			\item For even $n > 6\lfloor \sqrt{s+1}\rfloor-6$, we have 
			\begin{equation}\label{eq formula in Main thm}
				\mathtt{g}(t_{n}, t_{n+1}, t_{n+2}; s)
				=
				\frac{(n+1)(n+2)}{4}(q_s n + 6c_s) - 1.
			\end{equation}
			\item 
			For odd $n > 6\bigg\lfloor \frac{\sqrt{4s+5}-1}{2}\bigg\rfloor-3$, we have 
			\begin{equation} 
            \label{eq formula in Main thm odd}
				\mathtt{g}(t_{n}, t_{n+1}, t_{n+2}; s)
				=
				\frac{(n+1)(n+2)}{4}(q_s n + 6c_s - 3 \delta_s) - 1.
			\end{equation}
		\end{enumerate}
		Here the $q_s, c_s$ and $\delta_s$ are given by 
		\begin{align*}
			q_s = 
			2\lfloor \sqrt{s} \rfloor + 2 + \delta_s,
			\quad c_s =s - \lfloor \sqrt{s} \rfloor^2-\delta_s \lfloor \sqrt{s} \rfloor, \qquad
			 \delta_s =
			\begin{cases}
				1,
				&\text{ if } s \geq \lfloor \sqrt{s} \rfloor^2+\lfloor \sqrt{s} \rfloor,
				\\
				0,
				&\text{otherwise}.
			\end{cases}
		\end{align*}
	\end{theorem}
    
		For $s \geq 0$, we write for the bounds appearing in above theorem
\begin{equation}\label{definition N}
    N_s^{even} 
			\coloneqq 6\lfloor \sqrt{s+1}\rfloor-6 \qquad \text{ and } \qquad N_s^{odd} 
			\coloneqq
			6\bigg\lfloor \frac{\sqrt{4s+5}-1}{2}\bigg\rfloor-3.
\end{equation}
		
   We also define $$\mathbb{B} = \{n \in \mathbb{N} \mid n = k^2 \text{ or } n = k(k+1) \text{ for some } k \geq 1\} = \{1,2,4,6,9,12,16,\ldots\}.$$
	We will see later that if $s \notin \mathbb{B}$, then the condition of $n$ in Theorem \ref{thm-main-formula} becomes $n \geq N_{s}^{even}$ and $n \geq N_{s}^{odd}$.
	\begin{table}[h]
		\renewcommand{\arraystretch}{1.3}
		\centering
		\begin{tabular}{|c|c|c|c|c|c|c|c|c|c|c|c|c|c|c|c|c|c|c|c|c|c|}
			\hline
			$s$& 0 & 1 & 2 & 3 & 4 & 5 & 6 & 7 & 8 & 9 & 10 & 11 & 12 & 13 & 14 & 15 & 16 & 17 & 18 & 19 & 20 \\
			\hline
			
			$q_{s}$& 3 & 4 & 5 & 5 & 6 & 6 & 7 & 7 & 7 & 8 & 8 & 8 & 9 & 9 & 9 & 9 & 10 & 10	& 10 & 10 & 11\\
			
			$c_{s}$& 0 & 0 & 0 & 1 & 0 & 1 & 0 & 1 & 2 & 0 & 1 & 2 & 0 & 1 & 2 & 3 & 0 & 1 & 2 & 3 & 0\\
			
			$\delta_{s}$& 1 & 0 & 1 & 1 & 0 & 0 & 1 & 1 & 1 & 0 & 0 & 0 & 1 & 1 & 1 & 1 & 0 & 0 & 0 & 0 & 1\\
			\hline
		\end{tabular}
        \vspace{0.25em}
        \caption{$q_{s}$, $c_{s}$ and $\delta_{s}$ when $s = 0,\ldots,20$}
	\end{table}
	

    As a consequence of Theorem \ref{thm-main-formula} we will prove the following result, which gives a proof of Komatsu's conjecture.
	\begin{theorem}\label{thm-different-of-g} 
		Let $s ,n\in \mathbb{Z}_{\geq 0}$. Suppose that $n>N_{s+1}^{even}$ if $n$ is even and $n>N_{s+1}^{odd}$ if $n$ is odd. Then the following statements hold:
		\begin{enumerate}[(i)]
			\item If $s+1 \not\in \mathbb{B}$,  we have 
			\begin{equation*}
				\mathtt{g} (t_{n}, t_{n+1}, t_{n+2}; s+1)-		\mathtt{g} (t_{n}, t_{n+1}, t_{n+2}; s)
				= \frac{6(n+1)(n+2)}{4}.
			\end{equation*}
			
			\item If $n$ is even and $s+1 \in \mathbb{B}$, that is $ s+1 = k^{2}$ or $ s+1 =k(k+1)$  $(\exists k \geq 1)$, then
			\begin{equation*}
				\mathtt{g} (t_{n}, t_{n+1}, t_{n+2}; s+1)-		\mathtt{g} (t_{n}, t_{n+1}, t_{n+2}; s)
				= \frac{(n - 6k + 6)(n+1)(n+2)}{4}. 
			\end{equation*}
			
			\item If $n$ is odd and $s+1 \in \mathbb{B}$, then
			\begin{equation*}
				\mathtt{g} (t_{n}, t_{n+1}, t_{n+2}; s+1)-		\mathtt{g} (t_{n}, t_{n+1}, t_{n+2}; s)
				=
				\begin{cases}
					\frac{(n - 6k + 9)(n+1)(n+2)}{4}	&\text{ if } s+1 = k^{2},
					\\
					\frac{(n - 6k + 3)(n+1)(n+2)}{4}	&\text{ if } s+1 = k(k+1).
				\end{cases}
			\end{equation*}
		\end{enumerate}
	\end{theorem}
	The proof of Theorem \ref{thm-main-formula} and Theorem \ref{thm-different-of-g} are given in Section \ref{section-proof-thm-cor}.
	
	When $s=0$, we have $q_{s} = 3, c_{s} = 0$ and $ \delta_{0} = 1$. Our result comes back to the formulas in \cite{Roble_Rosales}. For $s = 1, \ldots, 10$, a formula for $\mathtt{g}(t_{n}, t_{n+1}, t_{n+2}; s)$ is given by \cite{Komatsu-triangular}.
	\begin{corollary}
		We have 
		\begin{align*}
                \mathtt{g} (t_{n}, t_{n+1}, t_{n+2}; 11)
			&=
                \frac{(n+1)(n+2)}{4}(8n+12) -1,
				\quad\text{for } n = 12 \text{ and } n\geq 14. 
			\\
			\mathtt{g} (t_{n}, t_{n+1}, t_{n+2}; 12)
			&=
			\begin{cases}
				\frac{(n+1)(n+2)}{4}(9n) -1,
				&\text{for even } n \geq 14,
				\\
				\frac{(n+1)(n+2)}{4}(9n-3) -1, 
                    &\text{for odd } n \geq 17,
			\end{cases}
			\\
			\mathtt{g} (t_{n}, t_{n+1}, t_{n+2}; 13)
			&=
			\begin{cases}
				\frac{(n+1)(n+2)}{4}(9n+6) -1, 
				&\text{for even } n\geq 12,
				\\
				\frac{(n+1)(n+2)}{4}(9n+3) -1,
                    &\text{for odd } n\geq 15,
			\end{cases}
			\\
			\mathtt{g} (t_{n}, t_{n+1}, t_{n+2}; 14)
			&=
			\begin{cases}
				\frac{(n+1)(n+2)}{4}(9n+12) -1, 
				&\text{for even } n\geq 12,
				\\
				\frac{(n+1)(n+2)}{4}(9n+9) -1,
                     &\text{for odd } n\geq 15,
			\end{cases}
			\\
			\mathtt{g} (t_{n}, t_{n+1}, t_{n+2}; 15)
			&=
			\begin{cases}
				\frac{(n+1)(n+2)}{4}(9n+18) -1,
				&\text{for even } n\geq 18,
				\\
				\frac{(n+1)(n+2)}{4}(9n+15) -1,
                    &\text{for odd } n\geq 15,
			\end{cases}
                \\
			\mathtt{g} (t_{n}, t_{n+1}, t_{n+2}; 16)
			&=
                \frac{(n+1)(n+2)}{4}(10n) -1,
				\quad\text{for } n= 17 \text{ and } n\geq 19,
                \\
			\mathtt{g} (t_{n}, t_{n+1}, t_{n+2}; 17)
			&=
                \frac{(n+1)(n+2)}{4}(10n+6) -1,
				\quad\text{for } n = 15 \text{ and } n\geq 17.
		\end{align*}
	\end{corollary}
	  
	\vspace{0.5cm}
	
	{\bf Acknowledgement:} 
    This project was supported by the Development and Promotion of Science and Technology Talents Project (DPST), Thailand. I sincerely appreciate the invaluable guidance and support of my supervisors, Prof. Henrik Bachmann and Prof. Kohji Matsumoto. I would also like to thank Prof. Takao Komatsu for his insightful comments.

	\section{A reformulation and preliminary Lemmas}
	In this section, we give a reformulation of Theorem \ref{thm-main-formula}, which is Theorem \ref{thm-g(tn,tn+1,tn+2;s)}. To facilitate the reformulation, we introduce the following variables: $x_{s}^{even}$, $y_{s}^{even}$, $x_{s}^{odd}$, and $y_{s}^{odd}$.
	
	\begin{definition}\label{definition x,y}
		Let $s$ be a non-negative integer and let $k$ be the non-negative integer such that $s = k(k+1) + i,$ for some $0 \leq i \leq 2k+1$. Then we define integers $x_{s}^{even}$, $y_{s}^{even}$, $x_{s}^{odd}$, and $y_{s}^{odd}$ as follows:
		\begin{align*}
			\big(x_{s}^{even}, y_{s}^{even}\big)
			&=
			\begin{cases}
				\big(i, 2(k-i)\big), &\text{ if } 0 \leq i \leq k,
				\\
				\big(i-k-1, 4k-2i+3\big), &\text{ if } k+1 \leq i \leq 2k+1,
			\end{cases}
			\\
			\big(x_{s}^{odd}, y_{s}^{odd}\big)
			&=
			\begin{cases}
				\big(2i, k-i\big), &\text{ if } 0 \leq i \leq k,
				\\
				\big(2(i-k)-1, 2k-i+1\big), &\text{ if }  k+1 \leq i \leq 2k+1,
			\end{cases}
		\end{align*}
	\end{definition}
		With the same parameters as defined in Definition \ref{definition x,y}, we can express $N_{s}^{even}$ and $N_{s}^{odd}$, introduced in \eqref{definition N}, as follows:
        \begin{equation*}
            N_s^{even} 
			=
			\begin{cases}
				6k - 6,  &\text{ if }  0 \leq i \leq k-1,
				\\
				6k,  &\text{ if } k \leq i \leq 2k+1,
			\end{cases}
            \qquad \text{ and } \qquad
            N_s^{odd} 
			=
			\begin{cases}
				6k - 3,  &\text{ if } 0 \leq i \leq 2k,
				\\
				6k+3,  &\text{ if } i = 2k+1.
			\end{cases}
        \end{equation*}
	
    The values of integers $x_{s}^{even}$, $y_{s}^{even}$, $x_{s}^{odd}$, and $y_{s}^{odd}$ are presented in Table \ref{tabel x,y}, while the values for $N_{s}^{even}$ and $N_{s}^{odd}$ can be found in Table \ref{tabel N}.
 
	\begin{table}
		\label{tabel x,y}
		\renewcommand{\arraystretch}{1.3}
		\centering
		\begin{tabular}{|c|c|c|c|c|c|c|c|c|c|c|c|c|c|c|c|c|c|c|c|c|}
			\hline
			$s$ & 0 & 1 & 2 & 3 & 4 & 5 & 6 & 7 & 8 & 9 & 10 & 11 & 12 & 13 & 14 & 15 & 16 & 17 & 18 & 19 \\
			\hline
			$x_{s}^{even}$ & 0 & 0 & 0 & 1 & 0 & 1 & 0 & 1 & 2 & 0 & 1 & 2 & 0 & 1 & 2 & 3 & 0 & 1 & 2 & 3 \\
			
			$y_{s}^{even}$ & 0 & 1 & 2 & 0 & 3 & 1 & 4 & 2 & 0 & 5 & 3 & 1 & 6 & 4 & 2 & 0 & 7 & 5 & 3 & 1\\
			\hline
			$x_{s}^{odd}$ & 0 & 1 & 0 & 2 & 1 & 3 & 0 & 2 & 4 & 1 & 3 & 5 & 0 & 2 & 4 & 6 & 1 & 3 & 5 & 7\\
			
			$y_{s}^{odd}$ & 0 & 0 & 1 & 0 & 1 & 0 & 2 & 1 & 0 & 2 & 1 & 0 & 3 & 2 & 1 & 0 & 3 & 2 & 1 & 0\\
			\hline
		\end{tabular}
        \vspace{0.25em}
        \caption{$x_{s}^{even}$, $y_{s}^{even}$, $x_{s}^{odd}$ and $y_{s}^{odd}$ when $s = k(k+1) + i$ for $0 \leq i \leq 2k+1$}
	\end{table}
	\begin{table}
		\renewcommand{\arraystretch}{1.3}
		\centering
		\begin{tabular}{|c|c|c|c|c|}
			\hline
			$i$ & $ 0,1,\ldots,k-1 $ & $k$ & $k+1, \ldots, 2k$ & $2k+1$ \\
			\hline
			$N_{s}^{even}$ & $6k-6$ & $6k$ & $6k$ & $6k$ \\
			\hline
			$N_{s}^{odd}$ & $6k-3$ & $6k-3$ & $6k-3$ & $6k+3$ \\
			\hline
		\end{tabular}
	\vspace{0.25em}
        \caption{$N_{s}^{even}$ and $N_{s}^{odd}$ for $k\geq 1$ when $s = k(k+1) + i$ for $0 \leq i \leq 2k+1$}
        \label{tabel N}
        \end{table}
	
	\begin{theorem}\label{thm-g(tn,tn+1,tn+2;s)}
		Let $s$ be a non-negative integer. If $s = 0,1,2$ or $s \notin \mathbb{B}$, then for integer $n\geq 2$
		\begin{equation*}
			\mathtt{g} (t_{n}, t_{n+1}, t_{n+2}; s)
			=
			\begin{cases}
				\frac{(n+1)(n+2)}{4}\Big((2x_{s}^{even} + y_{s}^{even} +3)n + 6x_{s}^{even}\Big) -1, 
				&\text{for even } n > N^{even}_{s},
				\\
				\frac{(n+1)(n+2)}{4}\Big((x_{s}^{odd} + 2y_{s}^{odd} +3)n + 3(x_{s}^{odd}-1)\Big) -1, &\text{for odd } n > N^{odd}_{s}.
			\end{cases}
		\end{equation*}
	\end{theorem}
	
	In the next section we will show how Theorem \ref{thm-g(tn,tn+1,tn+2;s)} implies Theorem \ref{thm-main-formula}. To prove Theorem \ref{thm-g(tn,tn+1,tn+2;s)} we will now give several lemmas.

	For a positive integer $n$, $t_{n} $ is the $n$th triangular number which is given by $t_{n} = \binom{n+1}{2} = \frac{n(n+1)}{2}.$ We also define 
	\begin{equation*}
		d_{1} \coloneqq \gcd(t_{n+1}, t_{n+2}) = 
		\begin{cases}
			\frac{n+2}{2}, 	&\text{ if } n \text{ is even,}\\
			n+2,			&\text{ if } n \text{ is odd.}
		\end{cases} 
	\end{equation*}
	Then
	\begin{center}
		$\dfrac{t_{n+1}}{d_{1}}
		=
		\begin{cases}
			n+1, 	&\text{ if } n \text{ is even,}\\
			\frac{n+1}{2},			&\text{ if } n \text{ is odd,}
		\end{cases}
		\quad$
		and
		$
		\quad
		\dfrac{t_{n+2}}{d_{1}}
		=
		\begin{cases}
			n+3, 	&\text{ if } n \text{ is even,}\\
			\frac{n+3}{2},			&\text{ if } n \text{ is odd.}
		\end{cases}
		$
	\end{center}
	
	Beck and Kifer \cite{Beck_Kifer} show the following result on $\mathtt{g}(a_{1}, a_{2}, \ldots, a_{k};s)$ in terms of $\ell = \gcd(a_{2}, a_{3}, \ldots, a_{k}) $.
	\begin{lemma}\label{lemma-from-Beck-Kifer}
		\cite[Lemma 4]{Beck_Kifer}
		For $k\geq 2$, let $A =(a_{1}, \ldots, a_{k})$ be a $k$-tuple of positive integers with $\gcd(A) = 1$. If $\ell = \gcd(a_{2}, a_{3}, \ldots, a_{k})$, let $a_{j} = \ell a_{j}'$ for $2 \leq j \leq k$. Then for $s\geq 0$
		\begin{align*}
			\mathtt{g}(a_{1}, a_{2}, \ldots, a_{k}; s) &= \ell \, \mathtt{g}\big(a_{1}, a_{2}', a_{3}', \ldots, a_{k}'; s\big) + a_{1}(\ell -1).
		\end{align*}
	\end{lemma}
        Our main approach involves applying Lemma \ref{lemma-from-Beck-Kifer} to three consecutive triangular numbers $t_{n}, t_{n+1},$ and $t_{n+2}$ by setting $\ell = d_{1} \coloneqq \gcd(t_{n+1}, t_{n+2})$. For $s \geq 0$, this leads to the expression: 
        \begin{equation*}
        \mathtt{g}(t_{n}, t_{n+1}, t_{n+2}; s) = d_{1}\mathtt{g}\Big(t_{n}, \frac{t_{n+1}}{d_{1}}, \frac{t_{n+2}}{d_{1}}; s\Big) + t_{n}(d_{1}-1).
        \end{equation*}
By Theorem \ref{thm: g(t1,t2d1,t3d1) = g(xs)+ystn}, it follows that
\begin{equation*}
    \mathtt{g}\Big(t_{n}, \frac{t_{n+1}}{d_{1}}, \frac{t_{n+2}}{d_{1}}; s\Big) = \mathtt{g}\Big(\frac{t_{n+1}}{d_{1}}, \frac{t_{n+2}}{d_{1}}; x_{s}\Big) + y_{s}t_{n},
\end{equation*}
where $(x_{s}, y_{s}) = (x_{s}^{\text{even}}, y_{s}^{\text{even}})$ for even $n > N_{s}^{\text{even}}$, and $(x_{s}, y_{s}) = (x_{s}^{\text{odd}}, y_{s}^{\text{odd}})$ for odd $n > N_{s}^{\text{odd}}$. This leads directly to the formulas presented in Theorem \ref{thm-g(tn,tn+1,tn+2;s)}.

In section \ref{section-proof-thm-cor}, we establish a connection between Theorem \ref{thm-g(tn,tn+1,tn+2;s)} and our main result in Theorem \ref{thm-main-formula}. The proof of these claims involves introducing several supporting lemmas.

	\begin{lemma}\label{lemma: formula d to sum of d_tilda}
		For integers $m \geq 0$,
		\begin{equation*}
			\operatorname{d}\!\Big( m ; t_{n}, \frac{t_{n+1}}{d_1}, \frac{t_{n+2}}{d_1}\Big)
			=
			\sum_{j=0}^{\lfloor m / t_n\rfloor} \operatorname{d}\!\Big(m - jt_{n} ;\frac{t_{n+1}}{d_1}, \frac{t_{n+2}}{d_1} \Big)
		\end{equation*}
	\end{lemma}
	\begin{proof}
		It follows immediately from the definition of $\operatorname{d}\!$.
	\end{proof}
    
    The detailed proofs of the following Lemmas are discussed and presented in \cite{Subwattanachai_g_three}. 
	\begin{lemma}\label{lemma: d(g_s-jc) = i iff g_i-1 < g_s-jc <= g_i}
    \cite[Lemma 5]{Subwattanachai_g_three}
		Let $a,b\in\mathbb{Z}_{>0}$ with $\gcd(a,b) = 1$, and let $i,s \in \mathbb{Z}_{\geq0}$. Suppose that $c$ is a positive integer such that $c  \equiv 0 \pmod a$ or $c \equiv 0 \pmod b$ and $j \in \mathbb{Z}$. Then
		\begin{equation*}
			\operatorname{d}\!\big( \mathtt{g}( a, b ;s) + jc ; a,b\big) = i,
		\end{equation*}
		if and only if,
		\begin{equation*}			
			\mathtt{g}(a, b ;i-1) < \mathtt{g}( a, b ;s) + jc \leq \mathtt{g}(a, b ;i).
		\end{equation*}
		Here we set $\mathtt{g}(a, b ;-1)$ to be $-\infty$.
	\end{lemma}
        \begin{lemma} \label{lemma for optimal m}
        \cite[Lemma 7]{Subwattanachai_g_three}
		Let $a,b \in \mathbb{Z}_{>0}$ with $a<b$, $\gcd(a,b) = 1$, and let $s, k \in \mathbb{Z}_{\geq 0}$. If m is an integer such that $m > \mathtt{g}(a,b;s) + ka$, then, for all $j \in \mathbb{Z}_{\geq 0}$, we have
		\begin{equation*}
			\operatorname{d}\!\big( m - ja; a,b\big) \geq \operatorname{d}\!\big(\mathtt{g}(a,b;s) + (k -j)a; a,b\big).
		\end{equation*}
	\end{lemma}
	The next lemma show that, for $s \geq 0$, the number of representations for $\mathtt{g}\Big(\frac{t_{n+1}}{d_1}, \frac{t_{n+2}}{d_1}; x_s\Big) + y_s t_n$ in terms of $t_{n}, \frac{t_{n+1}}{d_{1}},$ and $\frac{t_{n+2}}{d_{1}}$ is equal to $s$.
	
	\begin{lemma}\label{lem-main leamma for thm 1}
		Let $k$ be a non-negative integer and $i \in \{0,1, \ldots, 2k+1\}$. We let $s := s_{k,i} = k(k+1) + i$. Then, for each even $n > N_s^{even}$ and odd $n > N_s^{odd}$,
		\begin{equation}\label{eq: main leamma for thm 1}
			\operatorname{d}\! \Bigg( \mathtt{g}\Big(\frac{t_{n+1}}{d_1}, \frac{t_{n+2}}{d_1}; x_s\Big) + y_s t_n ; t_n, \frac{t_{n+1}}{d_1}, \frac{t_{n+2}}{d_1} \Bigg)  = k(k+1) + i = s,
		\end{equation}
		where $(x_{s},y_{s}) = (x_{s}^{even},y_{s}^{even})$ if $n$ is even and $(x_{s},y_{s}) = (x_{s}^{odd},y_{s}^{odd})$ if $n$ is odd.
	\end{lemma}
	\begin{proof}
		Throughout the proof, fo we denote 
		\begin{equation*}
			\mathtt{g}(m) := \mathtt{g}\Big(\frac{t_{n+1}}{d_1}, \frac{t_{n+2}}{d_1}; m\Big)
			\quad
			\text{ and } \quad
			\tilde{\operatorname{d}}(m) := \operatorname{d}\!\Big(m ; \frac{t_{n+1}}{d_1}, \frac{t_{n+2}}{d_1}\Big).
		\end{equation*}
		Then the left-handed side of \eqref{eq: main leamma for thm 1}, followed by Lemma \ref{lemma: formula d to sum of d_tilda}, can be written as
		\begin{equation}\label{eq in main Lemma}
			\operatorname{d}\! \Bigg( \mathtt{g}(x_{s}) + y_s t_n ; t_n, \frac{t_{n+1}}{d_1}, \frac{t_{n+2}}{d_1} \Bigg)
			=
			\sum_{j=0}^{\lfloor (\mathtt{g}(x_{s})) + y_s t_n )/t_{n} \rfloor} \tilde{\operatorname{d}}\big( \mathtt{g}(x_{s}) + (y_{s}-j)t_{n} \big).
		\end{equation}
		Recall again that, for the case 2 variables, $\tilde{\operatorname{d}}(\mathtt{g}(m)) = m$ for all $m \geq 0$.
		We will prove this lemma by induction on non-negative integers $k$.
		\\
		\textbf{Base step:} Let $k =0$. So $i \in\{0,1\}$. 
        
        Suppose that $i=0$, then $s = 0$. So $N_0^{even} = 0 $, $ N_0^{odd} = -3$ and $x_0^{even} = y_0^{even} = x_{0}^{odd} = y_{0}^{odd} = 0$. Since $n$ is a positive integer, it follows, by Lemma \ref{lemma: formula d to sum of d_tilda} and Lemma \ref{lemma: d(g_s-jc) = i iff g_i-1 < g_s-jc <= g_i}, that $\operatorname{d}\!\big(\mathtt{g}(0); t_{n}, \frac{t_{n+1}}{d_{1}}, \frac{t_{n+2}}{d_{1}}\big) = 0$.
		
		If $i =1$, then $s =1$, $N_{1}^{even} = 0$, $N_{1}^{odd} = 3$. For each even integer $n > 0$, we have
		\begin{center}
			$x_{1} = x_{1}^{even} = 0\quad$ and $\quad y_{1} = y_{1}^{even} = 1$.
		\end{center}
		Observe that, for all even $n > 0 = N_1^{even}$,
		\begin{equation*}
			\mathtt{g}(0) < \mathtt{g}(0)+ t_{n} \leq \mathtt{g}(1),
		\end{equation*}
		thus, by Lemma \ref{lemma: d(g_s-jc) = i iff g_i-1 < g_s-jc <= g_i},  $\tilde{\operatorname{d}}\big(\mathtt{g}(0)+ t_{n}\big) = 1$. So, for all even numbers $n > 0$,
		\begin{equation*}
			\operatorname{d}\! \Big(\mathtt{g}(0) + t_{n}; t_{n}, \frac{t_{n+1}}{d_{1}}, \frac{t_{n+2}}{d_{1}}\Big) = \tilde{\operatorname{d}}\big(\mathtt{g}(0)+t_{n}\big) + \sum_{j=0}^{\lfloor (\mathtt{g}(0)+t_{n})/t_{n}\rfloor} \tilde{\operatorname{d}}\big(\mathtt{g}(0)-jt_{n}\big) = 1.
		\end{equation*}
		For each odd integer $n > 3 = N_{1}^{odd},$ we have
		\begin{center}
			$x_{1} = x_{1}^{odd} = 1\quad$ and $\quad y_{1} = y_{1}^{odd} = 0$.
		\end{center}
		One can show that, for odd $n \geq 3$, $ \mathtt{g}(1)-t_{n} \leq \mathtt{g}(0).$
		Therefore, by \eqref{eq in main Lemma} and Lemma \ref{lemma: d(g_s-jc) = i iff g_i-1 < g_s-jc <= g_i}, we obtain that, for odd $n > 3$,
		\begin{equation*}
			\operatorname{d}\! \Big(\mathtt{g}(1); t_{n}, \frac{t_{n+1}}{d_{1}}, \frac{t_{n+2}}{d_{1}}\Big) = \tilde{\operatorname{d}}\big(\mathtt{g}(1)\big) + \sum_{j=1}^{\lfloor \mathtt{g}(1)/t_{n}\rfloor} \tilde{\operatorname{d}}\big(\mathtt{g}(1)-jt_{n}\big) = 1.
		\end{equation*}
		
		\textbf{Induction step:} Assume $k$ is a positive integer and we suppose that $P(k-1)$ holds. In other words, for $r \in \{0,1,\ldots, 2k-1\}$ with $v \coloneqq v_{r} = (k-1)k + r$, the following condition is satisfied:
		\begin{equation*}
			\operatorname{d}\!\Bigg(\mathtt{g}(x_{v}) + y_{v}t_{n}; t_{n}, \frac{t_{n+1}}{d_{1}}, \frac{t_{n+2}}{d_{1}}\Bigg)
			 = (k-1)k + r = v,
		\end{equation*}
		for even $n > N_{v}^{even}$ and odd $n > N_{v}^{odd}$. $N_{v}^{even}$ and $N_{v}^{odd}$ are shown in Table \ref{tabel N_r}. 
        
        The goal is to show that $P(k)$ holds. To prove this, we divide the argument into eight cases, as outlined in Table \ref{tabel N}. We begin with  the scenario when $n$ is even.
        \begin{table}[h]
			\renewcommand{\arraystretch}{1.3}
			\centering
			\begin{tabular}{|c|c|c|c|c|}
				\hline
				$r$ & $ 0,1,\ldots,k-2 $ & $k-1$ & $k, \ldots, 2k-2$ & $2k-1$ \\
				\hline
				$N_{v}^{even}$ & $6k-12$ & $6k-6$ & $6k-6$ & $6k-6$ \\
				\hline
				$N_{v}^{odd}$ & $6k-9$ & $6k-9$ & $6k-9$ & $6k-3$ \\
				\hline
			\end{tabular}
                \vspace{0.25em}
			\caption{$N_{v}^{even}$ and $N_{v}^{odd}$ for $k\geq 1$ when $v = (k-1)k + r$ for $0 \leq r \leq 2k-1$}
            \label{tabel N_r}
		\end{table}
		
		\textbf{Case E.1.} Let $i \in \{0,1,\ldots, k-1\}$ and $s = k(k+1) + i$. So $N_{s}^{even} = 6k-6$,
		\begin{center}
			$x_{s}^{even} = i \quad$ and $\quad y_{s}^{even} = 2(k-i)$.
		\end{center}
		Observe that, by induction hypothesis, we have, for each $r = 0,1,\ldots, k-1$ and $n > 6k-6$, $x_{v}^{even} = r$, and $y_{v}^{even} = 2((k-1)-r)$. So
		\begin{equation*}
			\sum_{j=0}^{\left\lfloor\frac{\mathtt{g}(r)}{t_{n}}\right\rfloor+2(k-r-1)} \tilde{\operatorname{d}}\Big(\mathtt{g}(r) + (2k-2r-2-j)t_{n}\Big) 
			= 
			\operatorname{d}\!\Bigg(\mathtt{g}(r)+ 2(k-1-r)t_{n}; t_{n}, \frac{t_{n+1}}{d_{1}}, \frac{t_{n+2}}{d_{1}} \Bigg)
			= (k-1)k + r.
		\end{equation*}
		So, it follows that
		\begin{align} \label{eq case E.1}
			&\operatorname{d}\!\Bigg(\mathtt{g}(i)+ 2(k-i)t_{n}; t_{n}, \frac{t_{n+1}}{d_{1}}, \frac{t_{n+2}}{d_{1}} \Bigg)
			\nonumber
			\\
			&=
			\tilde{\operatorname{d}}\Big(\mathtt{g}(i)+(2k-2i)t_{n}\Big) +
			\tilde{\operatorname{d}}\Big(\mathtt{g}(i)+(2k-2i-1)t_{n}\Big)
			+
			\sum_{j=0}^{\left\lfloor\frac{\mathtt{g}(i)}{t_{n}}\right\rfloor+2k-2i-2} \tilde{\operatorname{d}}\Big(\mathtt{g}(i) + (2k-2i-2-j)t_{n}\Big)
			\nonumber
			\\
			&=
			\tilde{\operatorname{d}}\Big(\mathtt{g}(i)+(2k-2i)t_{n}\Big) +
			\tilde{\operatorname{d}}\Big(\mathtt{g}(i)+(2k-2i-1)t_{n}\Big)
			+ (k-1)k + i.
		\end{align}
		We will claim that for even $n > 6k-6$,
		\begin{equation}\label{relation case E.1}
			\mathtt{g}(k-1) < \mathtt{g}(i) + (2k-2i-1)t_{n} < \mathtt{g}(i)+(2k-2i)t_{n} \leq \mathtt{g}(k).
		\end{equation}
		To determine the condition for an even integer $n$ such that $\mathtt{g}(k-1) < \mathtt{g}(i) + (2k-2i-1)t_{n}$, we can equivalently examine the inequality:
		\begin{equation*}
			\mathtt{g}(i)-\mathtt{g}(k-1) + (2k-2i-1)t_{n} > 0.
		\end{equation*}
		Since $n$ is even, $\mathtt{g}(m) = \mathtt{g}\big(\frac{t_{n+1}}{d_{1}}, \frac{t_{n+2}}{d_{1}}; m \big) = (m+1)(n+1)(n+3) - (n+1) - (n+3)$. Consequently, the above inequality becomes
		\begin{align*}
			&\hspace{1cm}(i-k+1)(n+1)(n+3) + (2k-2i-1)\frac{n(n+1)}{2} > 0
			\\
			&\Longleftrightarrow(2i-2k+2)(n+3) + (2k-2i-1)n > 0
			\\
			&\Longleftrightarrow n+ 6(i-k+1) > 0
			\quad\Longleftrightarrow \quad n > 6(k-i-1),
		\end{align*}
		Therefore, the inequality holds for even $n \geq 6k-6$ $(i = 1,\ldots,k-1)$ and holds for even $n >6k-6$ $(i=0)$.
		
        On the other hand, to show that $\mathtt{g}(i)+(2k-2i)t_{n} \leq \mathtt{g}(k)$ for even $n>6k-6$, it is equivalent to consider the inequality
		\begin{align*}
			&\hspace{1cm}(k-i)(n+1)(n+3) - (2k-2i)\frac{n(n+1)}{2} \geq 0
			\\
			&\Longleftrightarrow(2k-2i)(n+3) -(2k-2i)n \geq 0
			\quad\Longleftrightarrow\quad 6(k-i) \geq 0,
		\end{align*}
		which is always true since $i=0,1,\ldots,k-1$. By \eqref{relation case E.1} and Lemma \ref{lemma: d(g_s-jc) = i iff g_i-1 < g_s-jc <= g_i}, we get
		\begin{equation*}
			\tilde{\operatorname{d}}\Big(\mathtt{g}(i)+(2k-2i)t_{n}\Big) =
			\tilde{\operatorname{d}}\Big(\mathtt{g}(i)+(2k-2i-1)t_{n}\Big)
			= k.
		\end{equation*} 
		Hence, \eqref{eq case E.1} becomes
		\begin{equation*}
			\operatorname{d}\!\Bigg(\mathtt{g}(i)+ 2(k-i)t_{n}; t_{n}, \frac{t_{n+1}}{d_{1}}, \frac{t_{n+2}}{d_{1}} \Bigg)
			=
			2k + (k-1)k + i
			=
			k(k+1) + i = s.
		\end{equation*}
		
		\textbf{Case E.2.} Let $i = k$ and $s = k(k+1)+k$. In this case, $N_{s}^{even} = 6k$, Then $x_{s}^{even} = k$ and $y_{s}^{even} = 0$. So
		\begin{equation} \label{eq case E.2}
			\operatorname{d}\!\Bigg(\mathtt{g}(k); t_{n}, \frac{t_{n+1}}{d_{1}}, \frac{t_{n+2}}{d_{1}} \Bigg)
			=
			\tilde{\operatorname{d}}\Big(\mathtt{g}(k)\Big) +
			\sum_{j=1}^{\left\lfloor\frac{\mathtt{g}(k)}{t_{n}}\right\rfloor} \tilde{\operatorname{d}}\Big(\mathtt{g}(k) - jt_{n}\Big)
			=
			k + 	\sum_{j=1}^{\left\lfloor\frac{\mathtt{g}(k)}{t_{n}}\right\rfloor} \tilde{\operatorname{d}}\Big(\mathtt{g}(k) - jt_{n}\Big).
		\end{equation}
		We will claim that, for all $n > 6k$ and $\ell = 0,1,\ldots, k$,
		\begin{equation}\label{relation in E.2}
			\mathtt{g}\big(k-\ell-1\big) < \mathtt{g}(k) - (2\ell+2)t_{n} < \mathtt{g}(k) - (2\ell+1)t_{n} \leq \mathtt{g}\big(k-\ell\big).
		\end{equation}
		The second inequality is obvious. To establish the first condition, we examine
		\begin{align*}
			\mathtt{g}\big(k-\ell-1\big) < \mathtt{g}(k) - (2\ell+2)t_{n}
			&\Longleftrightarrow (\ell+1)(n+1)(n+3) - (2\ell+2)t_{n} >0
			\\
			&\Longleftrightarrow (2\ell+2)(n+3) - (2\ell+2)n >0
			\\
			&\Longleftrightarrow 6(\ell+1) >0,
		\end{align*}
		which holds for $\ell \geq 0$. In the same way, we have that
		\begin{align*}
			\mathtt{g}(k) - (2\ell+1)t_{n} \leq \mathtt{g}\big(k-\ell\big)
			&\Longleftrightarrow n \geq 6\ell,
		\end{align*}
		which is true since $n > 6k \geq 6\ell$. Thus $\eqref{relation in E.2}$ holds. Additionally,  it is not challenging to demonstrate that $ \mathtt{g}(k) \geq 2kt_{n}$ for all $n \geq 1$. By \eqref{relation in E.2} and Lemma \ref{lemma: d(g_s-jc) = i iff g_i-1 < g_s-jc <= g_i}, \eqref{eq case E.2} becomes
		\begin{equation*}
			\operatorname{d}\!\Bigg(\mathtt{g}(k); t_{n}, \frac{t_{n+1}}{d_{1}}, \frac{t_{n+2}}{d_{1}} \Bigg)
			=
			k + 2\big(k + (k-1) + \cdots + 1\big)
			=
			k + k(k+1) = s.
		\end{equation*}
		
		\textbf{Case E.3.} Let $i \in\{k+1, k+2, \ldots, 2k\}$ and let $s = k(k+1)+i$. In this case we have $N_{s}^{even} = 6k$,
		\begin{center}
			$x_{s}^{even} = i-(k+1)\quad$ and $\quad y_{s}^{even} = 4k-2i+3$.
		\end{center}
		By the induction hypothesis, we have, for each $r \in \{k, k+1, \ldots, 2k-1\}$, $v \coloneqq (k-1)k+r$, $x_{v}^{even} = r-(k-1+1) = r-k$ and $y_{v}^{even} = 4(k-1)-2r+3$. Then, for even $n > 6k-6$,
		\begin{align*}
			(k-1)k + r 
			&=
			\operatorname{d}\!\Bigg(\mathtt{g}(r-k)+ \big(4(k-1)-2r+3\big)t_{n}; t_{n}, \frac{t_{n+1}}{d_{1}}, \frac{t_{n+2}}{d_{1}} \Bigg)
			\\
			&=
			\sum_{j=0}^{\left\lfloor\frac{\mathtt{g}(r-k)}{t_{n}}\right\rfloor+(4k-2r-1)} \tilde{\operatorname{d}}\Big(\mathtt{g}(r-k) + (4k-2r-1-j)t_{n}\Big).
		\end{align*}
		Observe that
		\begin{align} \label{eq case E.3}
			&\operatorname{d}\!\Bigg(\mathtt{g}(i-(k+1))+ (4k-2i+3)t_{n}; t_{n}, \frac{t_{n+1}}{d_{1}}, \frac{t_{n+2}}{d_{1}} \Bigg)
			\nonumber
			\\
			&=
			\tilde{\operatorname{d}}\Big(\mathtt{g}(i-k-1)+(4k-2i+3)t_{n}\Big) +
			\tilde{\operatorname{d}}\Big(\mathtt{g}(i-k-1)+(4k-2i+2)t_{n}\Big)
			\nonumber
			\\
			&\hspace{2cm}+
			\sum_{j=0}^{\left\lfloor\frac{\mathtt{g}(i-k-1)}{t_{n}}\right\rfloor+4k-2i+1} \tilde{\operatorname{d}}\Big(\mathtt{g}(i-k-1) + (4k-2i+1-j)t_{n}\Big).
		\end{align}
	One can see that, in the last summation, because $i-1 \in \{k, k+1, \ldots, 2k-1\}$, then by the induction hypothesis we obtain
		\begin{equation*}
			\sum_{j=0}^{\left\lfloor\frac{\mathtt{g}(i-k-1)}{t_{n}}\right\rfloor+4k-2i+1} \tilde{\operatorname{d}}\Big(\mathtt{g}(i-k-1) + (4k-2i+1-j)t_{n}\Big)
            = (k-1)k + i -1.
		\end{equation*}
		Therefore, \eqref{eq case E.3} becomes
		\begin{align*}
			\operatorname{d}\!\Bigg(&\mathtt{g}(i-(k+1))+(4k-2i+3)t_{n}; t_{n}, \frac{t_{n+1}}{d_{1}}, \frac{t_{n+2}}{d_{1}} \Bigg)
			\\
			&=
			\tilde{\operatorname{d}\!}\Big(\mathtt{g}(i-k-1)+(4k-2i+3)t_{n}\Big) +
			\tilde{\operatorname{d}\!}\Big(\mathtt{g}(i-k-1)+(4k-2i+2)t_{n}\Big)
			\\
			&\hspace{10pt}+
			(k-1)k+i-1.
		\end{align*}
		We claim that for even $n > 6k$
		\begin{equation}\label{relation in E.3.1}
			\mathtt{g}(k-1) < \mathtt{g}(i-k-1) + (4k-2i+2)t_{n}  \leq \mathtt{g}(k)	
		\end{equation} 
		and
		\begin{equation}\label{relation in E.3.2}
			\mathtt{g}(k) < \mathtt{g}(i-k-1) + (4k-2i+3)t_{n}  \leq \mathtt{g}(k+1).
		\end{equation}
		Since the inequality
		\begin{align*}
			\mathtt{g}(k-1) < \mathtt{g}(i-k-1) + (4k-2i+2)t_{n}
			&\Longleftrightarrow
			n > 3(2k-i).
		\end{align*}
		It holds for $i = k+1, \ldots, 2k$. In the same way, we have
		\begin{align*}
			\mathtt{g}(i-k-1) + (4k-2i+2)t_{n}  \leq \mathtt{g}(k)	
			&\Longleftrightarrow
			3(4k-2i+2) \geq 0,
		\end{align*}
		which is always true. So \eqref{relation in E.3.1} holds for $n > 6k$. To show \eqref{relation in E.3.2}, we examine the inequality
		\begin{align*}
			\mathtt{g}(k) < \mathtt{g}(i-k-1) + (4k-2i+3)t_{n}
			&\Longleftrightarrow
			n >  
            6(2k-i+1),
		\end{align*}
		which holds since $i \geq k+1$ and $n > 6k$. Notice that it still holds for even $n \geq 6k$ if $i = k+2,\ldots, 2k-1$ (excluding $i = k+1$). On the other hand,
		\begin{align*}
			\mathtt{g}(i-k-1) + (4k-2i+3)t_{n}  \leq \mathtt{g}(k+1)
			&\Longleftrightarrow
			n + 6(2k-i+2) > 0,
		\end{align*}
		which is always true since $i \leq 2k$. So \eqref{relation in E.3.2} holds. Therefore
		\begin{equation*}
			\operatorname{d}\!\Bigg(\mathtt{g}(i-(k+1))+(4k-2i+3)t_{n}; t_{n}, \frac{t_{n+1}}{d_{1}}, \frac{t_{n+2}}{d_{1}} \Bigg)
			=
			(k+1) + k + (k-1)k+i-1
			= k(k+1) + i = s.
		\end{equation*}
		
		\textbf{Case E.4.} Let $i = 2k+1$ and $s = k(k+1)+ (2k+1)$. Then $N_{s}^{even} = 6k$,
		\begin{center}
			$x_{s}^{even} = 2k+1-(k+1) = k \quad$ and $\quad y_{s}^{even} = 4k - 2(2k+1) +3 = 1$ 
		\end{center}
		Then, by \eqref{relation in E.2} in Case E.2, it follows that for all even $n > 6k$
		\begin{align}\label{eq in Case E.4}
			\operatorname{d}\!\Bigg(\mathtt{g}(k)+t_{n}; t_{n}, \frac{t_{n+1}}{d_{1}}, \frac{t_{n+2}}{d_{1}} \Bigg)
			&=
			\tilde{\operatorname{d}}\Big(\mathtt{g}(k)+t_{n}\Big) +
			\tilde{\operatorname{d}}\Big(\mathtt{g}(k)\Big) +
			\sum_{j=1}^{\left\lfloor\frac{\mathtt{g}(k)}{t_{n}}\right\rfloor} \tilde{\operatorname{d}}\Big(\mathtt{g}(k) - jt_{n}\Big)
			\nonumber
			\\
			&= \tilde{\operatorname{d}}\Big(\mathtt{g}(k)+t_{n}\Big) +
			k +
			k(k+1).
		\end{align}
		It remains to show that $ \tilde{\operatorname{d}}\Big(\mathtt{g}(k)+t_{n}\Big) = k+1$ which is equivalent to show that
		\begin{equation*}
			\mathtt{g}(k) < \mathtt{g}(k)+t_{n} \leq \mathtt{g}(k+1).
		\end{equation*}
		The first condition is obvious and we have
		\begin{align*}
			\mathtt{g}(k)+t_{n} \leq \mathtt{g}(k+1) 
                &\Longleftrightarrow
            n \geq -6.
		\end{align*}
		Therefore, for even $n > 6k$, \eqref{eq in Case E.4} becomes
		\begin{equation*}
			\operatorname{d}\!\Bigg(\mathtt{g}(k)+t_{n}; t_{n}, \frac{t_{n+1}}{d_{1}}, \frac{t_{n+2}}{d_{1}} \Bigg)
			= (k+1)+ k + k(k+1) = k(k+1) + (2k+1) = s.
		\end{equation*}
		Hence, we have completed the proof for the case when $n > N_{s}^{even}$ is even. Lastly, we will prove the case when $n > N_{s}^{odd}$ is odd. 
		
		\textbf{Case O.1.} Let $i \in \{0,1,\ldots, k-1\}$ and let $s = k(k+1)+i$. In this case, $N_{s}^{odd} = 6k-3$,
		\begin{center}
			$x_{s}^{odd} = 2i\quad$ and $\quad y_{s}^{odd} = k-i$.
		\end{center}
		By induction hypothesis, we have, for all $r \in\{0,1,\ldots, k-1\}$ with $v = (k-1)k +r$, $x_v^{odd} = 2r$, $y_{v}^{odd} = k-1-r$ so that, for odd $n > 6k-9$, 
		\begin{equation*}
			(k-1)k + r = 
			\operatorname{d}\!\Bigg(\mathtt{g}(2r)+ (k-1-r)t_{n}; t_{n}, \frac{t_{n+1}}{d_{1}}, \frac{t_{n+2}}{d_{1}} \Bigg)
			=
			\sum_{j=0}^{\left\lfloor\frac{\mathtt{g}(2r)}{t_{n}}\right\rfloor+(k-r-1)} \tilde{\operatorname{d}}\Big(\mathtt{g}(2r) + (k-r-1-j)t_{n}\Big).
		\end{equation*}
		Observe that
		\begin{align} \label{eq case O.1}
			&\operatorname{d}\!\Bigg(\mathtt{g}(2i)+ (k-i)t_{n}; t_{n}, \frac{t_{n+1}}{d_{1}}, \frac{t_{n+2}}{d_{1}} \Bigg)
			\nonumber
			\\
			&=
			\tilde{\operatorname{d}}\Big(\mathtt{g}(2i)+(k-i)t_{n}\Big) 
			+
			\sum_{j=0}^{\left\lfloor\frac{\mathtt{g}(2i)}{t_{n}}\right\rfloor+k-i-1} \tilde{\operatorname{d}}\Big(\mathtt{g}(2i) + (k-i-1-j)t_{n}\Big)
			\nonumber
			\\
			&=
			\tilde{\operatorname{d}}\Big(\mathtt{g}(2i)+(k-i)t_{n}\Big)
			+ (k-1)k + i.
		\end{align}
		We need to show that for odd $n > 6k-3$
		\begin{equation}\label{relation in O.1}
			\mathtt{g}(2k-1) < \mathtt{g}(2i) + (k-i)t_{n} \leq \mathtt{g}(2k).
		\end{equation}
		Since $n$ is odd, $\mathtt{g}(m) = \mathtt{g}\big(\frac{t_{n+1}}{d_{1}}, \frac{t_{n+2}}{d_{1}}; m\big) = (m+1)\frac{(n+1)(n+3)}{4} - \frac{(n+1)}{2} - \frac{(n+3)}{2}$. Thus
		\begin{align*}
			\mathtt{g}(2k-1) < \mathtt{g}(2i) + (k-i)t_{n} 
			&\Longleftrightarrow
			(2i-2k+1)\frac{(n+1)(n+3)}{4} + (k-i)t_{n} > 0
			\\
			&\Longleftrightarrow
			(2i-2k+1)(n+3) + (2k-2i)n > 0
			\\
			&\Longleftrightarrow
			n  > 6k-6i-3,
		\end{align*}
		which holds since $i \geq 0$ and $n>6k-3$. Notice that if $i = 1,\ldots, k-1$ (excluding $i = 0$), it also holds for all odd $n \geq 6k-3$. On the other hand, 
		\begin{align*}
			\mathtt{g}(2i) + (k-1)t_{n} \leq \mathtt{g}(2k) 
			&\Longleftrightarrow
			3(2k-2i) > 0,
		\end{align*}
		which is always true since $0 \leq i \leq k-1$. Therefore, by \eqref{relation in O.1}, \eqref{eq case O.1}, and Lemma \ref{lemma: d(g_s-jc) = i iff g_i-1 < g_s-jc <= g_i}, we obtain
		\begin{equation*}
			\operatorname{d}\!\Bigg(\mathtt{g}(2i)+ (k-i)t_{n}; t_{n}, \frac{t_{n+1}}{d_{1}}, \frac{t_{n+2}}{d_{1}} \Bigg)
			= 2k + (k-1)k +i = k(k+1)+i =s.
		\end{equation*}
		
		\textbf{Case O.2.} Let $i = k$ and $s = k(k+1)+k$. Then $N_{s}^{odd} = 6k-3$, $x_{s}^{odd} = 2k$ and $y_{s}^{odd} = 0$.
		We first claim that for all $\ell \in \{1,\ldots, k\}$ and for odd $n > 6k-3$
		\begin{equation}\label{relation in Case O.2}
			\mathtt{g}(2k-2\ell) < \mathtt{g}(2k)-\ell t_{n} \leq \mathtt{g}(2k-2\ell+1)
		\end{equation}
		One can show that
		\begin{align*}
			\mathtt{g}(2k-2\ell) < \mathtt{g}(2k)-\ell t_{n} 
			&\Longleftrightarrow 
			6\ell  > 0,
		\end{align*}
		which is true for all $\ell \geq 1$. On the other hand, we obtain that
		\begin{align*}
			\mathtt{g}(2k)- \ell t_{n} \leq \mathtt{g}(2k-2\ell+1)
			&\Longleftrightarrow 
			n    \geq 6\ell-3,
		\end{align*}
		which holds since $n > 6k-3$. Observe that $\mathtt{g}(2k) \geq kt_{n}$ for $n \geq 1$. Therefore, by applying Lemma \ref{lemma: d(g_s-jc) = i iff g_i-1 < g_s-jc <= g_i} to \eqref{relation in Case O.2}, we obtain
		\begin{align*}
			\operatorname{d}\!\Bigg(\mathtt{g}(2k); t_{n}, \frac{t_{n+1}}{d_{1}}, \frac{t_{n+2}}{d_{1}} \Bigg)
			&=
			\tilde{\operatorname{d}}\Big(\mathtt{g}(2k)\Big)  
			+
			\sum_{j=1}^{\left\lfloor\frac{\mathtt{g}(2k)}{t_{n}}\right\rfloor} \tilde{\operatorname{d}}\Big(\mathtt{g}(2k) - jt_{n}\Big)
			\\
			&= 2k + \Big((2k-1) + (2k-3) + \cdots + 3 + 1\Big) = 2k + k^{2} = s.
		\end{align*}
		
		\textbf{Case O.3.} Let $i \in \{k+1, \ldots, 2k\}$ and let $s = k(k+1) + i$. In this case we have $N_{s}^{odd} = 6k-3$,
		\begin{center}
			$x_{s}^{odd} = 2(i-(k+1))+1 = 2i-2k-1 \quad$ and $\quad y_{s}^{odd} = 2k-i+1$.
		\end{center}
		By the induction hypothesis, if $r \in \{k, k+1, \ldots, 2k-1\}$ and $v \coloneqq (k-1)k + r$, then $x_{v}^{odd} = 2(r-k)+1 = 2r-2k+1$ and $y_{v}^{odd} = 2(k-1)-r+1 = 2k-r-1$. Then, for odd $n > 6k-3 = N_{v}^{odd}$, we have
		\begin{align*}
			(k-1)k + r 
			&=
			\operatorname{d}\!\Bigg(\mathtt{g}(2r-2k+1)+ (2k-r-1)t_{n}; t_{n}, \frac{t_{n+1}}{d_{1}}, \frac{t_{n+2}}{d_{1}} \Bigg)
			\\
			&=
			\sum_{j=0}^{\left\lfloor\frac{\mathtt{g}(2r-2k+1)}{t_{n}}\right\rfloor + 2k-r-1} \tilde{\operatorname{d}}\Big(\mathtt{g}(2r-2k+1)+ (2k-r-1-j)t_{n} \Big).
		\end{align*}
		We consider
		\begin{align*}
			&\operatorname{d}\!\Bigg(\mathtt{g}(2i-2k-1)+ (2k-i+1)t_{n}; t_{n}, \frac{t_{n+1}}{d_{1}}, \frac{t_{n+2}}{d_{1}} \Bigg)
			\\
			&=
			\tilde{\operatorname{d}}\Big(\mathtt{g}(2i-2k-1)+ (2k-i+1)t_{n} \Big)
			+
			\sum_{j=0}^{\left\lfloor\frac{\mathtt{g}(2i-2k-1)}{t_{n}}\right\rfloor + 2k-i} \tilde{\operatorname{d}}\Big(\mathtt{g}(2i-2k-1)+ (2k-i-j)t_{n} \Big).
		\end{align*}
		Notice that $i-1 \in \{k, k+1,\ldots, 2k-1\}$. So, in the last summation, we obtain by the induction hypothesis that the right-hand side of the above equation is equal to
		\begin{equation*}
			\tilde{\operatorname{d}}\Big(\mathtt{g}(2i-2k-1)+ (2k-i+1)t_{n} \Big)
			+
			(k-1)k + i - 1.
		\end{equation*}
	So it remains to show that for odd $n > 6k-3$
		\begin{equation}\label{relation in O.3}
			\mathtt{g}(2k) < \mathtt{g}(2i-2k-1) + (2k-i+1)t_{n} \leq \mathtt{g}(2k+1).
		\end{equation}
		Again, we have
		\begin{align*}
			\mathtt{g}(2k) < \mathtt{g}(2i-2k-1) + (2k-i+1)t_{n} 
			&\Longleftrightarrow 
			n  > 3(4k-2i+1),
		\end{align*}
		which is true since $i \geq k+1$ and $n > 6k-3$. Notice that it also holds for $n \geq 6k-3$ if $i \in \{k+2, k+3,\ldots, 2k\}$. In the same way,
		\begin{align*}
			\mathtt{g}(2i-2k-1) + (2k-i+1)t_{n} \leq \mathtt{g}(2k+1)
			&\Longleftrightarrow 
			6(2k-i+1)  \geq 0,
		\end{align*}
		which holds for all $i = k+1, k+2, \ldots, 2k$ and all $n$. This implies \eqref{relation in O.3} holds, so we are done in this case. 
		
		\textbf{Case O.4.} On the last case, we let $i = 2k+1$ and let $s = k(k+1) + 2k+1 = (k+1)(k+2) - 1$. We have $N_{s}^{odd} = 6k+3$, $x_{s}^{odd} = 2k+1$ and $ y_{s}^{odd} = 0$.
		
		We first show that, for odd $n > 6k+3$, if $\ell \in \{1,2,\ldots, k+1\}$,
		\begin{equation}\label{relation in O.4}
			\mathtt{g}\big(2k-2\ell+1\big) < \mathtt{g}(2k+1) - \ell t_{n} \leq \mathtt{g}\big(2k-2\ell+2\big)
		\end{equation}
		So, we consider
		\begin{align*}
			\mathtt{g}(2k-2\ell+1) < \mathtt{g}(2k+1) - \ell t_{n}
			&\Longleftrightarrow 
			6\ell  > 0,
		\end{align*}
		which is always true for $\ell \geq 1$. On the other hand, we have
		\begin{align*}
			\mathtt{g}(2k+1) - \ell t_{n} \leq \mathtt{g}\big(2k-2\ell+2 \big)
			&\Longleftrightarrow 
			n \geq 6\ell-3,
		\end{align*}
		which holds for $\ell = 1,2,\ldots, k+1$ since $n > 6k+3 = 6(k+1)-3$.
		
		One can also shows that $\mathtt{g}(2k+1) \geq kt_{n}$ for all $n \geq 1$. Using \eqref{relation in O.4} and Lemma \ref{lemma: d(g_s-jc) = i iff g_i-1 < g_s-jc <= g_i}, we obtain
		\begin{align*}
			\operatorname{d}\!\Bigg(\mathtt{g}(2k+1); t_{n}, \frac{t_{n+1}}{d_{1}}, \frac{t_{n+2}}{d_{1}} \Bigg)
			&=
			\tilde{\operatorname{d}}\Big(\mathtt{g}(2k+1)\Big)
			+
			\sum_{j=1}^{\left\lfloor\frac{\mathtt{g}(2k+1)}{t_{n}}\right\rfloor} \tilde{\operatorname{d}}\Big(\mathtt{g}(2k+1) -jt_{n} \Big)
			\\
			&=
			(2k+1) + \big(2k + 2(k-1) + \cdots + 2 \big)
			\\
			&= (2k+1) + k(k+1) = s,
		\end{align*}
		for all odd integers $n > 6k+3 = N_{s}^{odd}$. Hence we complete to prove the induction step.
	\end{proof}
	
		Note that, regarding the proof of Lemma \ref{lem-main leamma for thm 1}, if $s \notin \mathbb{B}$ (i.e. $i \neq 0, k+1$), the formula \eqref{eq: main leamma for thm 1} holds for even $n \geq N_s^{even}$ and odd $n \geq N_s^{odd}$.
	Next, we present Theorem \ref{thm: g(t1,t2d1,t3d1) = g(xs)+ystn}, which is a consequence of Lemma \ref{lem-main leamma for thm 1} and Lemma \ref{lemma for optimal m}. Following this, we show the proof for Theorem \ref{thm-g(tn,tn+1,tn+2;s)}.
	
	\begin{theorem}\label{thm: g(t1,t2d1,t3d1) = g(xs)+ystn}
		Let $s$ be a non-negative integer with $s= k(k+1)+i$ for some integers $k \geq 0$ and $0 \leq i \leq 2k+1$. Then, for  even $n > N_{s}^{even}$ and odd $n > N_{s}^{odd}$, 
		\begin{align*}
			\mathtt{g}\Big(t_{n}, \frac{t_{n+1}}{d_{1}}, \frac{t_{n+2}}{d_{1}}; s\Big)
			&
			= \mathtt{g}\Big(\frac{t_{n+1}}{d_{1}}, \frac{t_{n+2}}{d_{1}}; x_{s}\Big) + y_{s}t_{n}
			\nonumber \\
			&
			=
			(x_{s} +1)\frac{t_{n+2}t_{n+1}}{d_{1}^{2}} - \frac{t_{n+2}}{d_{1}} - \frac{t_{n+1}}{d_{1}} + y_{s}t_{n},
		\end{align*}
		where $(x_{s},y_{s}) = (x_{s}^{even}, y_{s}^{even})$ if $n$ is even and $(x_{s},y_{s}) = (x_{s}^{odd}, y_{s}^{odd})$ if $n$ is odd.
	\end{theorem}
	\begin{proof}
		Again, for convenient, we denote
		\begin{equation*}
			\mathtt{g}(m) := \mathtt{g}\Big(\frac{t_{n+1}}{d_1}, \frac{t_{n+2}}{d_1}; m\Big)
			\quad
			\text{ and } \quad
			\tilde{\operatorname{d}}(m) := \operatorname{d}\!\Big(m ; \frac{t_{n+1}}{d_1}, \frac{t_{n+2}}{d_1}\Big).
		\end{equation*}
		By Lemma \ref{lem-main leamma for thm 1}, we have $	\operatorname{d}\! \bigg( \mathtt{g}(x_{s})+ y_s t_n ; t_n, \frac{t_{n+1}}{d_1}, \frac{t_{n+2}}{d_1} \bigg) = s$. Therefore, by the definition of $\mathtt{g}\Big(t_{n}, \frac{t_{n+1}}{d_{1}}, \frac{t_{n+2}}{d_{1}}; s\Big)$, it follows that
		\begin{equation*}
			\mathtt{g}(x_{s}) + y_{s}t_{n}
			\leq 
			\mathtt{g}\Big(t_{n}, \frac{t_{n+1}}{d_{1}}, \frac{t_{n+2}}{d_{1}}; s\Big).
		\end{equation*}
		Suppose that $m$ is an integer such that $m > 	\mathtt{g}(x_{s}) + y_{s}t_{n}$. Then $m-y_{s}t_{n} > \mathtt{g}(x_{s})$ and
		\begin{equation*}
			\tilde{\operatorname{d}}\Big(m - y_{s}t_{n} \Big) =	\operatorname{d}\!\Big(m - y_{s}t_{n} ;\frac{t_{n+1}}{d_1}, \frac{t_{n+2}}{d_1} \Big) > x_{s} = \operatorname{d}\!\Big( \mathtt{g}\Big(\frac{t_{n+1}}{d_{1}}, \frac{t_{n+2}}{d_{1}}; x_{s}\Big) ;\frac{t_{n+1}}{d_1}, \frac{t_{n+2}}{d_1} \Big) = \tilde{\operatorname{d}}\Big( \mathtt{g}(x_{s}) \Big)
		\end{equation*}
		Since $t_{n+1}/d_{1}$ divides $t_{n}$ for all positive integers n, following the result in Lemma \ref{lemma for optimal m} we have that, for $j \in \mathbb{Z}_{\geq 0}$,
        \begin{equation*}
        \tilde{\operatorname{d}}\Big( m-jt_n \Big) \geq \tilde{\operatorname{d}}\Big( \mathtt{g}(x_s)+(y_s -j)t_n \Big).
        \end{equation*}
        So, by comparing term by term, we obtain that
		\begin{align*}
			\operatorname{d}\!\Big(m; t_{n}, \frac{t_{n+1}}{d_{1}}, \frac{t_{n+2}}{d_{1}}\Big)
			&=
			\sum_{j=0}^{\lfloor m / t_n\rfloor} \tilde{\operatorname{d}}\Big(m - jt_{n} \Big)
			\geq
			\sum_{j=0}^{\lfloor \mathtt{g}(x_{s}) / t_n\rfloor +y_{s}} \tilde{\operatorname{d}}\Big(m - jt_{n} \Big)
			\\
			&>
			\sum_{j=0}^{\lfloor \mathtt{g}(x_{s}) / t_n\rfloor +y_{s}} \tilde{\operatorname{d}}\Big(\mathtt{g}(x_{s}) + (y_{s}- j)t_{n} \Big)
			=
			\operatorname{d}\!\Big(\mathtt{g}(x_{s}) + y_{s}t_{n}; t_{n}, \frac{t_{n+1}}{d_{1}}, \frac{t_{n+2}}{d_{1}}\Big) = s.
		\end{align*}
		This implies that 
		\begin{equation*}
			\mathtt{g}\Big(t_{n}, \frac{t_{n+1}}{d_{1}}, \frac{t_{n+2}}{d_{1}}; s\Big)
			= \mathtt{g}\Big(\frac{t_{n+1}}{d_{1}}, \frac{t_{n+2}}{d_{1}}; x_{s}\Big) + y_{s}t_{n}
			=
			(x_{s} +1)\frac{t_{n+2}t_{n+1}}{d_{1}^{2}} - \frac{t_{n+2}}{d_{1}} - \frac{t_{n+1}}{d_{1}} + y_{s}t_{n}.
            \tag*{\qedhere}
		\end{equation*}
	\end{proof}
	We now ready to prove Theorem \ref{thm-g(tn,tn+1,tn+2;s)}.
	\begin{proof}[Proof of Theorem \ref{thm-g(tn,tn+1,tn+2;s)}]
		By Lemma \ref{lemma-from-Beck-Kifer} and Theorem \ref{thm: g(t1,t2d1,t3d1) = g(xs)+ystn}, we have, for $s \geq 0$,
		\begin{align*}
			\mathtt{g} (t_{n}&, t_{n+1}, t_{n+2}; s)
			= d_{1}\mathtt{g}\Big(t_{n}, \frac{t_{n+1}}{d_{1}}, \frac{t_{n+2}}{d_{1}} ; s\Big) + t_{n}(d_{1}-1)
			\\
			&= 
			\begin{cases}
				(x^{even}_{s} +1) \frac{t_{n+1}t_{n+2}}{d_{1}} + (y^{even}_{s}+1)t_{n}d_{1} - t_{n} -t_{n+1} - t_{n+2}, &\text{for even } n > N^{even}_{s},
				\\
				(x^{odd}_{s} +1) \frac{t_{n+1}t_{n+2}}{d_{1}} + (y^{odd}_{s}+1)t_{n}d_{1} - t_{n} -t_{n+1} - t_{n+2},
				&\text{for odd } n > N^{odd}_{s}.
			\end{cases}
			\\
			&=
			\begin{cases}
				\frac{(n+1)(n+2)}{4}\Big((2x_{s}^{even} + y_{s}^{even} +3)n + 6x_{s}^{even}\Big) -1, 
				&\text{for even } n > N^{even}_{s},
				\\
				\frac{(n+1)(n+2)}{4}\Big((x_{s}^{odd} + 2y_{s}^{odd} +3)n + 3x_{s}^{odd}-3\Big) - 1,
				&\text{for odd } n > N^{odd}_{s}.
            \end{cases}
            \tag*{\qedhere}	
		\end{align*}
	\end{proof}

		Similarly to Lemma \ref{lem-main leamma for thm 1}, Theorem \ref{thm-g(tn,tn+1,tn+2;s)} and Theorem \ref{thm: g(t1,t2d1,t3d1) = g(xs)+ystn} also hold for all even numbers $n \geq N_{s}^{even}$ and odd numbers $n \geq N_{s}^{odd}$ when $s \notin \mathbb{B}$.

		\section{Proof of Theorem \ref{thm-main-formula} and Theorem \ref{thm-different-of-g}}\label{section-proof-thm-cor}
		
		\begin{proof}[Proof of Theorem \ref{thm-main-formula}]
			Let $s \geq 0$ be an integer. Assume that $s = k(k+1) + i$ for some $k\geq 0$ and $i\in \{0,1,\ldots,2k+1\}$. 
			
			\textbf{Case 1.} If $i \in\{0,1,\ldots, k\}$, then, by Definition \ref{definition x,y} and Theorem \ref{thm-g(tn,tn+1,tn+2;s)}, we have
			\begin{center}
				$\big(x_{s}^{even}, y_{s}^{even}\big) = \big(i , \, 2(k-i)\big)\quad$
				and
				$\quad\big(x_{s}^{odd}, y_{s}^{odd}\big) = \big(2i , \, k-i\big)$,
			\end{center}
			and
			\begin{align*}
				\mathtt{g} (t_{n},& t_{n+1}, t_{n+2}; s)
				\\
				&=
				\begin{cases}
					\frac{(n+1)(n+2)}{4}\Big((2i + 2(k-i) +3)n + 6i\Big) -1 ,
					&\text{for even }  n> N^{even}_{s},
					\\
					\frac{(n+1)(n+2)}{4}\Big((2i + 2(k-i) +3)n + 3(2i)-3\Big) - 1,
					&\text{for odd } n> N^{odd}_{s}.
				\end{cases}
				\\
				&=
				\begin{cases}
					\frac{(n+1)(n+2)}{4}\Big((2k+3)n + 6i\Big) -1 ,
					&\text{for even } n> N^{even}_{s},
					\\
					\frac{(n+1)(n+2)}{4}\Big((2k+3)n + 6i-3\Big) - 1,
					&\text{for odd } n> N^{odd}_{s}.
				\end{cases}
			\end{align*}
			In this case, one can see that $\lfloor \sqrt{s}\rfloor = k$ and $\delta_s = 1$ since $ s =k(k+1) + i \geq k^{2}+k$. It follows that $$q_{s} = 2\lfloor \sqrt{s}\rfloor+2+\delta_s = 2k+2+1 = 2k+3,$$  
			$$\quad c_s = s - \lfloor \sqrt{s}\rfloor^{2} - \delta_{s}\lfloor \sqrt{s}\rfloor  = k(k+1) + i - k^{2} - k= i.$$ Then the formula \eqref{eq formula in Main thm} is proven for this case. 
			
			\textbf{Case 2.} If $i \in\{k+1, k+2,\ldots, 2k+1\}$, then, by Definition \ref{definition x,y} and Theorem \ref{thm-g(tn,tn+1,tn+2;s)}, we have
			\begin{center}
				$\big(x_{s}^{even}, y_{s}^{even}\big) = \big(i-k-1 , \, 4k-2i+3\big)\quad$
				and
				$\quad\big(x_{s}^{odd}, y_{s}^{odd}\big) = \big(2(i-k)-1 , \, 2k-i+1\big)$,
			\end{center}
			and
			\begin{align*}
				&\mathtt{g} (t_{n}, t_{n+1}, t_{n+2}; s)
				\\
				&=
				\begin{cases}
					\frac{(n+1)(n+2)}{4}\Big(\big((2i-2k-2) + (4k-2i+3) +3\big)n + 6(i-k-1)\Big) -1 ,
					&\text{for even } n> N^{even}_{s},
					\\
					\frac{(n+1)(n+2)}{4}\Big(\big((2i-2k-1) + (4k-2i+2) +3\big)n + (6i-6k-3)-3\Big) - 1,
					&\text{for odd } n> N^{odd}_{s}.
				\end{cases}
				\\
				&=
				\begin{cases}
					\frac{(n+1)(n+2)}{4}\Big((2k+4)n + 6(i-k-1)\Big) -1 ,
					&\text{for even } n> N^{even}_{s},
					\\
					\frac{(n+1)(n+2)}{4}\Big((2k+4)n + 6(i-k-1)\Big) - 1,
					&\text{for odd } n> N^{odd}_{s}.
				\end{cases}
			\end{align*}
			
			Since $s = k(k+1) + i \geq k(k+1)+(k+1) = (k+1)^{2}$, it follows that $ \lfloor \sqrt{s}\rfloor = k+1$ but $\delta_s = 0$ since $ s =k(k+1) + i < (k+1)^{2} + (k+1) =  \lfloor \sqrt{s}\rfloor^{2} + \lfloor \sqrt{s}\rfloor$. In this case, we have 
			$$q_{s} = 2\lfloor \sqrt{s}\rfloor+2+\delta_s = 2(k+1)+2+0 = 2k+4,$$
			$$\quad c_s = s - \lfloor \sqrt{s}\rfloor^{2} - \delta_{s}\lfloor \sqrt{s}\rfloor = k(k+1) + i -(k+1)^{2} = i-k-1.
			$$ 			
			Then we obtain immediately the formula \ref{eq formula in Main thm odd}. We finish the proof of Theorem \ref{thm-main-formula}.
		\end{proof}
		
		Next, we give the proof of Theorem \ref{thm-different-of-g}.
		
		\begin{proof}[Proof of Theorem \ref{thm-different-of-g}]
			Let $s \geq 0$. We divided the proof into 2 case; $s+1 \notin\mathbb{B}$ and  $s+1 \in\mathbb{B}$.
			
			\textbf{Case $s+1 \notin\mathbb{B}$.} Suppose that $s \geq 2$ and there exist an integer $k \geq 1$ such that $s = k(k+1)+i$,
			for some $i \in \{0,1,\ldots, 2k+1\}$. Since $s+1 \notin\mathbb{B}$, $s+1$ is not a square and it is not of the form $k(k+1)$. Then we have that either
			\begin{align*}
				k(k+1) &\leq s < s+1 \leq k(k+1)+k 
				\\
				\text{ or }\quad k(k+1)+ (k+1) &\leq s < s+1 \leq k(k+1)+(2k+1).
			\end{align*}
			By the proof of Theorem \ref{thm-main-formula}, both cases have $\lfloor \sqrt{s+1}\rfloor = \lfloor \sqrt{s}\rfloor$ and $\delta_{s+1} = \delta_s$. Thus $q_{s+1} = 2\lfloor \sqrt{s+1}\rfloor +2+\delta_{s+1} = 2\lfloor \sqrt{s}\rfloor +2+\delta_{s} = q_{s}$ and
			\begin{align*}
				&\mathtt{g} (t_{n}, t_{n+1}, t_{n+2}; s+1)-		\mathtt{g} (t_{n}, t_{n+1}, t_{n+2}; s)
				\\
				&= \frac{(n+1)(n+2)}{4}\Big( (q_{s+1} - q_{s})n + 6(c_{s+1}-c_{s})\Big)
				\\
				&= 
				\frac{(n+1)(n+2)}{4}\Bigg(6\bigg((s+1 - \lfloor\sqrt{s+1}\rfloor^{2}- \delta_{s+1}\lfloor\sqrt{s+1}\rfloor)-(s- \lfloor\sqrt{s}\rfloor^{2}- \delta_{s}\lfloor\sqrt{s}\rfloor)\bigg)\Bigg).
			\end{align*}
			Therefore
			\begin{equation*}
				\mathtt{g} (t_{n}, t_{n+1}, t_{n+2}; s+1)-		\mathtt{g} (t_{n}, t_{n+1}, t_{n+2}; s)
				= \frac{6(n+1)(n+2)}{4}.
			\end{equation*}
			
			\textbf{Case $s+1 \in\mathbb{B}$.} There exists an integer $k \geq 1$ such that one of the following statements holds: There exists $k \in \mathbb{Z}_{\geq1}$ such that
			\begin{enumerate}[(i)]
				\item $s+1 = k^{2} = (k-1)k + k\quad$ and $\quad s = k^{2}-1 = (k-1)k + (k-1)$ \label{case i in cor-proof-komatsu}
				\item  $s+1 = k(k+1) = (k-1)k + 2k\quad$  and $\quad s = k(k+1)-1 = (k-1)k + (2k-1)$ \label{case ii in cor-proof-komatsu}
			\end{enumerate}
			If $n$ is even and \eqref{case i in cor-proof-komatsu} holds, then, following the proof of in Theorem \ref{thm-main-formula}, we obtain that $ \lfloor \sqrt{s+1}\rfloor = k$, $\lfloor \sqrt{s}\rfloor = k-1$, $\delta_{s+1} = 0$,  and $\delta_{s} = 1$. So
			\begin{align*}
				q_{s+1} - q_{s} 
				&= \big(2\lfloor\sqrt{s+1}\rfloor+2+\delta_{s+1}\big) - \big(2\lfloor\sqrt{s}\rfloor+2+\delta_{s}\big)
				\\
				&= \big(2k+2\big) - \big(2(k-1)+2+1\big) = 1
			\end{align*}
			and
			\begin{align*}
				c_{s+1} - c_{s} 
				&= \big(s+1 - \lfloor\sqrt{s+1}\rfloor^{2} -  \delta_{s+1}\lfloor\sqrt{s+1}\rfloor\big)-\big(s-\lfloor\sqrt{s}\rfloor^{2}-\delta_s\lfloor\sqrt{s}\rfloor\big)
				\\
				&= \big(s+1 - k^{2}\big)-\big(s-(k-1)^{2} - (k-1)\big)
				\\
				&= 1 -k^{2}+k(k-1) = 1 - k.
			\end{align*}
			If $n$ is even and \eqref{case ii in cor-proof-komatsu} holds, we obtain	$\lfloor\sqrt{s+1}\rfloor = k  = \lfloor\sqrt{s}\rfloor$,
			$\delta_{s+1} = 1$ and $\delta_{s} = 0$. In this case, we have
			$$
			q_{s+1} - q_{s} = (2k+2+1)-(2k+2+0) = 1
			$$
			and
			$$
			c_{s+1} - c_{s} = \big(s+1 - k(k+1)\big)-\big(s-k^{2}\big)
			= 1 - k.
			$$ Hence, if $n$ is even and $s+1 \in \mathbb{B}$, then
			\begin{align*}
				\mathtt{g} (t_{n}, &t_{n+1}, t_{n+2}; s+1)-		\mathtt{g} (t_{n}, t_{n+1}, t_{n+2}; s)
				\\
				&= \frac{(n+1)(n+2)}{4}\Big( (q_{s+1} - q_{s})n + 6(c_{s+1}-c_{s})\Big) = \frac{(n+1)(n+2)}{4}\Big(n -6k + 6\Big).
			\end{align*}
			
			If $n$ is odd and \eqref{case i in cor-proof-komatsu} holds, similarly to the case $n$ is even, we have $\lfloor\sqrt{s+1}\rfloor = k$, $ = \lfloor\sqrt{s}\rfloor = k-1$, $\delta_{s+1} = 0$ and $\delta_{s} = 1$. Thus $q_{s+1} - q_{s} = 1$, $c_{s+1} - c_{s} = 1 - k$ and $\delta_{s+1} - \delta_{s} = -1 $.
			
			If $n$ is odd and \eqref{case ii in cor-proof-komatsu} holds, we have $\lfloor\sqrt{s+1}\rfloor = k = \lfloor\sqrt{s}\rfloor$, $\delta_{s+1} = 1$ and $\delta_{s} = 0$. Thus $q_{s+1} - q_{s} = 1$, $c_{s+1} - c_{s} = 1 - k$  and $\delta_{s+1} - \delta_{s} = 1 - 0 = 1$. 
			
			Hence, if $n$ is odd and $s+1 \in \mathbb{B}$, then
			\begin{align*}
				\mathtt{g} (t_{n}, &t_{n+1}, t_{n+2}; s+1)-		\mathtt{g} (t_{n}, t_{n+1}, t_{n+2}; s)
				\\&=
				\frac{(n+1)(n+2)}{4}\big((q_{s+1}-q_{s})n + 6(c_{s+1} - c_{s}) - 3(\delta_{s+1}-\delta_{s})\big) 
				\\
				&=\begin{cases}
					\frac{(n - 6k + 9)(n+1)(n+2)}{4}	&\text{ if } s+1 = k^{2},
					\\
					\frac{(n - 6k + 3)(n+1)(n+2)}{4}	&\text{ if } s+1 = k(k+1).
				\end{cases}
			\end{align*}
			This completes the proof of Theorem \ref{thm-different-of-g}.
		\end{proof}

\begin{remark}
  We introduce Theorem \ref{thm: g(t1,t2d1,t3d1) = g(xs)+ystn}, a reformulation of Theorem \ref{thm-g(tn,tn+1,tn+2;s)} by using the sequence $(x_s,y_s)$ from Definition \ref{definition x,y}. We can extend it by observing that, in the proof of Theorem \ref{thm: g(t1,t2d1,t3d1) = g(xs)+ystn}, its result can be extended beyond looking only at triangular numbers to any three integers that satisfy certain conditions as follows:
  
  Let $s$ be a positive integer and assume that $s = k(k+1) + i$ for some $i \in \{0,1,\ldots, 2k+1\}$ and let $A_{1}, A_{2}, A_{3} \in\mathbb{Z}_{> 1}$ such that $\gcd(A_{1},A_{2},A_{3}) = 1 $, $A_{1} \equiv 0\pmod{A_{2}}$. Let $K_{s}^{ev} = \lfloor\sqrt{s+1}\rfloor -1$ and $ K_{s}^{od} = \lfloor\frac{\sqrt{4s+5}-1}{2}\rfloor$. 

  If $2 < \frac{A_{2}A_{3}}{A_{1}}$ and $\frac{A_{2}A_{3}}{A_{1}} < 2 + \frac{1}{K_{s}^{ev}}$ (if $k\geq1$,i.e. $s \geq 2$), then
  \begin{equation*}
    \mathtt{g}\Big(A_{1}, A_{2}, A_{3}; s\Big) = \mathtt{g}\Big( A_{2}, A_{3}; x_{s}^{even}\Big) + y_{s}^{even}A_{1}.
  \end{equation*}

If $\frac{1}{2} < \frac{A_{2}A_{3}}{A_{1}} < 1$ ( if $k=0$) and $ \frac{1}{2} < \frac{A_{2}A_{3}}{A_{1}} < \frac{K_{s}^{od}}{2K_{s}^{od}-1}$ (if $k \geq 1$), then
    \begin{equation*}
        \mathtt{g}\bigg(A_{1}, A_{2}, A_{3}; s\bigg) = \mathtt{g}\bigg( A_{2}, A_{3}; x_{s}^{odd}\bigg) + y_{s}^{odd}A_{1}.
    \end{equation*}
These results can be proved in the same way as in Theorem \ref{thm: g(t1,t2d1,t3d1) = g(xs)+ystn} by induction proving and using some effort to consider the relations in each step. 
\end{remark}

		
        

\begin{thebibliography}{}
			
			
			
			\bibitem{Beck_Gessel_Komatsu}M. Beck, I. M. Gessel and T. Komatsu, \emph{The polynomial part of a restricted partition function related to the Frobenius problem}, Electron. J.
			Combin. \textbf{8} (No.1) (2001), $\#$N7.
			
			\bibitem{Beck_Kifer}M. Beck and C. Kifer, \emph{An extreme family of generalized Frobenius numbers}, Integers \textbf{11} (2011), A24, 639–645.
			
			\bibitem{Beck_Robins}M. Beck and S. Robins, \emph{A formula related to the Frobenius problem in two dimensions}.  Number theory (New York Seminar 2003), pages 17-23, Springer, NewYork, 2004.
			
			\bibitem{Binner-number of solutions to axby+cz=n}D.  S.  Binner,  \emph{The number of solutions to $ax + by+ cz =n$ and its relation to quadratic residues},  J.  Integer  Seq.  \textbf{23} (2020), Article 20.6.5.
			
			\bibitem{Binner-binner2021bounds} D.  S.  Binner, \emph{Some Bounds for Number of Solutions to $ax + by + cz = n$ and their Applications}, arXiv preprint \href{https://arxiv.org/abs/2106.13796}{arXiv:2106.13796}, 2021.
			
			
			
			\bibitem{Cayley:double_partition} A. Cayley, \emph{On a problem of double partitions}, Philos. Mag. \textbf{XX} (1860), 337–341.
			
			\bibitem{Fukshansky-Schurmann-Bounds_on} L. Fukshansky and A. Schurmann, Bounds on generalized Frobenius numbers, Eur. J. Comb., \textbf{32} (2011), 361–368. \href{https://doi.org/10.1016/j.ejc.2010.11.001}{https://doi.org/10.1016/j.ejc.2010.11.001}.
			
			
			\bibitem{Komatsu-On_the_number_of_solutions} T. Komatsu, \emph{On the number of solutions of the Diophantine equation	of Frobenius–General case}, Math. Commun. \textbf{8} (2003), 195–206.
			
			
			\bibitem{Komatsu-The Frobenius number repunits} T. Komatsu, \emph{The Frobenius number associated with the number of representations for sequences of repunits}, C. R. Acad. Sci. Paris \textbf{361} (2023), 73-89.
			
			\bibitem{Komatsu-triangular} T. Komatsu, \emph{The Frobenius number for sequences of triangular numbers associated with number of solutions}, Ann.	Comb. \textbf{26} (2022), no. 3, p. 757-779.

            \bibitem{Komatsu-Ying-Frob_Fibonacci}T. Komatsu and H. Ying, \emph{The $p$-Frobenius and $p$-Sylvester numbers for Fibonacci and Lucas triplets}.  Math.  Biosci. Eng.  \textbf{20} (2023),  3455--3481. 
			
			

                \bibitem{Subwattanachai_g_three} K. Subbwattanachai, \emph{Generalized Frobenius Number of Three Variables}, arXiv preprint \href{https://arxiv.org/abs/2309.09149}{arXiv:2309.09149} 2023.
			
			\bibitem{Roble_Rosales} A. M. Robles-P\'{e}rez and J. C. Rosales, \emph{The Frobenius number for sequences of triangular and tetrahedral numbers}, J. Number Theory \textbf{186} (2018), 473–492.
			
			
			
			
			\bibitem{Sylvester-On_the_partition_of_numbers} J. J. Sylvester, \emph{On the partition of numbers}, Quart. J. Pure Appl. Math.	\textbf{1} (1857), 141–152.
			
			\bibitem{Sylvester-On_subinvariants} J. J. Sylvester, \emph{On subinvariants, i.e. semi-invariants to binary quantics of an unlimited order}, Amer. J. Math. \textbf{5} (1882), 119–136.
			
			
			
			\bibitem{Tripathi-The_number_of_solutions} A. Tripathi, \emph{The number of solutions to $ax + by = n$}, Fibonacci Quart.	\textbf{38} (2000), 290–293.
			
			
			\bibitem{Tripathi-Formulae_for_Frobenius_three_variables}A. Tripathi, \emph{Formulae for the Frobenius number in three variables}, J.	Number Theory \textbf{170} (2017), 368–389.
			
			\bibitem{Woods-woods2022generalized} K. Woods, \emph{The generalized Frobenius problem via restricted partition functions}. arXiv preprint \href{https://arxiv.org/abs/2011.00600}{arXiv:2011.00600}, 2022.
		\end{thebibliography}
	\end{document}